\documentclass[11pt]{article}
\usepackage{amsthm,amsfonts}
\usepackage{amsmath}
\usepackage{amssymb}
\usepackage{graphics,epsfig}
\usepackage{stackrel}
\usepackage{color}
\newtheorem{theorem}{Theorem}[section]

\newtheorem{lemma}{Lemma}[section]

\newtheorem{definition}{Definition}[section]

\newtheorem{example}{Example}[section]
\newenvironment{remark}{\addtocounter{theorem}{1}\vskip 0.2cm{\sc Remark
\thetheorem.}}{\hfill\vskip 0.2cm}

\newcommand{\sign}{\operatornamewithlimits{sign}}

\newcommand{\var}{\operatornamewithlimits{\bf var}}
\newcommand{\beq}{\begin{equation}}
\newcommand{\enq}{\end{equation}}

\newcommand{\beqa}{\begin{eqnarray}}
\newcommand{\enqa}{\end{eqnarray}}
\newcommand{\beqas}{\begin{eqnarray*}}
\newcommand{\enqas}{\end{eqnarray*}}
\newcommand{\R}{{\mathbb{R}}}
\newcommand{\F}{{\mathcal{F}}}

\newcommand{\E}{ {\bf E}}
\newcommand{\Prob}{ {\bf P}}

\numberwithin{equation}{section}

\providecommand{\keywords}[1]{\textbf{{Keywords:}} #1}
\providecommand{\subclass}[1]{\textbf{{Mathematics Subject Classification 2000:}} #1}
\begin{document}
\title{Utility  maximization in Wiener-transformable markets}

\author{E. Boguslavskaya\footnote{
              Department of  Mathematics, Brunel University London, Uxbridge UB8 3PH, UK ,
             {elena@boguslavsky.net}}\,\,\thanks{Elena Boguslavskaya is supported by Daphne Jackson fellowship funded by EPSRC}\,\,\,
             \and
        Yu. Mishura\footnote{Department of Probability Theory, Statistics and Actuarial Mathematics,
 Taras Shevchenko National University of Kyiv,
64, Volodymyrs'ka St.,
 01601 Kyiv, Ukraine,           {myus@univ.kiev.ua}}}
 \date{}
 \maketitle
\abstract{
We consider a utility maximization problem in a broad class of markets. Apart from  traditional semimartingale markets, our class of markets includes processes with long memory, fractional Brownian motion and related processes, and, in general,
  Gaussian processes   satisfying certain regularity conditions on their covariance functions. Our choice of  markets is motivated by  the well-known phenomena of  the so-called ``constant'' and ``variable depth''  memory observed in real world price processes,  for which fractional and multifractional models are the most adequate descriptions, see, e.g., \cite{bian,Vehel}.
  We introduce the notion of a Wiener-transformable  Gaussian process, and give examples of such processes, and their representations.
The representation for
     the solution of the utility maximization problem in our specific setting is presented for various utility functions.
}
\\
\\
\\
\\
 \subclass{91B25;  91B16;   60G22}
\tableofcontents

\section{Introduction}
Many papers and books are written on utility maximization in semimartingale financial markets. Here we mention only the paper \cite{Bia-Fritt} with extended references therein, the book \cite{Foll-Sch}, where the general setting is described in very simple terms, and one of the most recent papers, \cite{cara-raso}, where an optimal investment problem is studied for a behavioral investor in an incomplete
discrete-time multiperiod financial model.

In the present  paper  we consider a utility maximization problem for a broader class of asset prices processes. We assume the asset prices to follow Gaussian processes subject to certain regularity conditions on their covariance functions. This class of processes includes processes with long memory, fractional Brownian motion and related processes. Numerous examples of such processes are provided.

To construct the capital process from a Gaussian price process $G$, we consider a strategy $\psi=\{\psi(t), t\in[0,T]\}$ for which the integral $\int_0^T\psi(t)dG(t)$ exists as a pathwise integral.  The capital with respect to the strategy $\psi$  at time $T$ is given by $d + \int_0^T\psi(t)dG(t),$ where the initial capital $d$ can be any real number.  In our opinion this is the simplest way to construct the capital in such a general setting.

The market can admit arbitrage, and moreover, without any additional restrictions, starting from any fixed initial value, we can acquire an arbitrary value of the final capital on such a market. However, even in such conditions, under some reasonable additional restrictions, the problem of utility maximization makes sense. The restriction involves the Radon-Nikodym derivative of the pricing measure.  In the standard Black-Scholes setting the pricing measure coincides with the unique martingale measure.

We are assuming that the Gaussian process $G$ generates the same filtration as a certain Wiener process, and combine the following three  facts in our approach.
\begin{itemize}
\item 
 Firstly, the Radon-Nikodym derivative $\varphi(T)$ of the pricing measure can be presented as the final value of some positive martingale, and under certain regularity assumptions this martingale is a process with H\"{o}lder trajectories.
\item Secondly, the solution of the utility maximization problem can be presented as a smooth function of the  pricing Radon-Nikodym derivative $\varphi(T)$, and, consequently, is  the final value of a H\"{o}lder process. The corresponding theorems are proved, e.g., in \cite{Foll-Sch}.
\item Thirdly, the final value of some H\"{o}lder process can be presented as  a pathwise integral $\int_0^T\psi(t)dG(t)$, or as $d+\int_0^T\psi(t)dG(t)$ for any constant $d\in\R$. It means that we can achieve the desirable maximal capital starting from any point. We provide a construction for the appropriate strategy.
\end{itemize}    
    Therefore, in some sense,   the main purpose of this paper is  to draw attention to the fact that in general markets without transaction costs, a utility maximization problem makes sense with a non-standard additional restriction.
    
    Markets with transaction costs as well as the reduction of  the solution of the utility maximization problem in our class of markets  to the corresponding partial differential equation will be the subjects of our future papers. Note that the strategy can be non-unique, thus one may hope that the  construction of the strategy proposed in the paper can be simplified.
    
    The paper is organized as follows. Section \ref{sec:2} contains the elements of fractional and Malliavin calculus and provides  the martingale and Clark-Ocone representations. In Section \ref{sec:3} we give the notion, examples   and representations  of Wiener-transformable  Gaussian processes that are used as the underlying price processes in our financial markets. These processes include a broad class of non-standard price processes. We formulate and comment on the representation result for  pathwise integrals w.r.t. such processes from \cite{mish-shev}, and prove an auxiliary result concerning H\"{o}lder properties of  stochastic It\^{o} integrals and their quadratic characteristics. Section \ref{sec:4} contains the solution of the utility maximization problem for the unrestricted capital under exponential utility. We use results from previous sections and from \cite{Foll-Sch} to demonstrate that the utility maximisation problem is well-posed. The uniqueness of the solution follows from convexity properties.  Some recommendations concerning the choice of $\varphi(T)$ are presented. Section \ref{sec:5} contains similar results for the restricted payoffs. Section \ref{sec:6} concludes with the description of an optimal strategy.

\section{Elements of fractional and Malliavin calculus. Martingale and Clark-Ocone representation}\label{sec:2}
We start with some preliminary definitions and representations.
\subsection{Elements of fractional calculus and fractional integration}
Here we present the basic facts on fractional integration; for more details see \cite{samko,zahle}. Consider functions  $f,g:[0,T]\rightarrow \mathbb{R}$, and let  $[a,b]\subset [0,T]$.
For $\alpha\in (0,1)$ define Riemann-Liouville fractional derivatives on finite interval $[a,b]$
\begin{gather*}
\big(\mathcal{D}_{a+}^{\alpha}f\big)(x)=\frac{1}{\Gamma(1-\alpha)}\bigg(\frac{f(x)}{(x-a)^\alpha}+\alpha
\int_{a}^x\frac{f(x)-f(u)}{(x-u)^{1+\alpha}}du\bigg)1_{(a,b)}(x),\end{gather*}
\begin{gather}\label{equ:dif}\big(\mathcal{D}_{b-}^{ \alpha}g\big)(x)=\frac{1} {\Gamma(1-\alpha)}\bigg(\frac{g(x)}{(b-x)^{ \alpha}}+ \alpha
\int_{x}^b\frac{g(x)-g(u)}{(u-x)^{1+\alpha}}du\bigg)1_{(a,b)}(x).
\end{gather}
Assuming that
 $\mathcal{D}_{a+}^{\alpha}f\in L_1[a,b]$, $\mathcal{D}_{b-}^{1-\alpha}g_{b-}\in
L_\infty[a,b]$, where $g_{b-}(x) = g(x) - g(b)$,
the generalized Lebesgue--Stieltjes
integral
is defined as
\begin{equation*}\int_a^bf(x)dg(x)= \int_a^b\big(\mathcal{D}_{a+}^{\alpha}f\big)(x)
\big(\mathcal{D}_{b-}^{1-\alpha}g_{b-}\big)(x)dx.
\end{equation*}

Let  function $g$ have H\"{o}lder trajectories, namely, $g\in C^\theta[a,b]$ with $\theta\in(\frac12,1)$. In order to integrate w.r.t. function  $g$ and to find an upper bound of the integral, fix some $\alpha \in(1-\theta,1/2)$ and introduce the following norm:
\begin{gather*}
\|f\|_{\alpha,[a,b]} = \int_a^b \left(\frac{|{f(s)}|}{(s-a)^\alpha} + \int_a^s \frac{|{f(s)-f(z)}|}{(s-z)^{1+\alpha}}dz\right)ds.
\end{gather*}
For simplicity we abbreviate $\|\cdot\|_{\alpha,t} = \|\cdot\|_{\alpha,[0,t]}$. Denote $$\Lambda_\alpha(g):= \sup_{0\le s<t\le T} |{\mathcal{D}_{t-}^{1-\alpha}g_{t-}}(s)|.$$ In view of H\"{o}lder continuity, $\Lambda_\alpha(g)<\infty$.

Then for any $t\in(0,T]$ and for any $f$ with $\|f\|_{\alpha,t}<\infty$, the integral $\int_0^t f(s) dg(s)$ is well defined as a generalized Lebesgue--Stieltjes integral, and the following bound is evident:
\begin{gather}\label{equ:ineq}
\Big|{\int_0^t f(s)dg(s)}\Big|\le \Lambda_\alpha(g) \|f\|_{\alpha,t}.
\end{gather}
It is well known that in the case of  both functions $f$ and  $g$ being  H\"{o}lder, more precisely,  $f\in C^\beta[a,b]$, $g\in C^\theta[a,b]$ with $\beta+\theta>1$, the generalized Lebesgue--Stieltjes
integral $\int_a^bf(x)dg(x)$ exists, equals to the limit of Riemann sums and admits bound \eqref{equ:ineq} for any $\alpha \in(1-\theta, \beta\wedge 1/2)$.
\begin{definition}
Let $\alpha>0.$
 The (right-sided) Riemann--Liouville fractional integral operator of
order $\alpha$ over the real line is defined by
$$
(\mathcal{I}^\alpha_{-}f )(s) :=\frac{1}{\Gamma(\alpha)}\int_s^{\infty} f(u)(u-s)^{\alpha-1}du,\;s\in\R.
$$
Let $T > 0$. The (right-sided) Riemann--Liouville fractional integral operator of order $\alpha$ over
$[0, T ]$ is defined by
$$(\mathcal{I}^\alpha_{T_{-}}f )(s)=(\mathcal{I}^\alpha_{-}f\mathbb{I}_{[0,T]} )(s),\; s>0$$.
\end{definition}
\begin{definition} Let $\alpha\in(0, 1)$. The (right-sided) Riemann--Liouville fractional derivative
operator of order $\alpha$ over the real line is defined by
\begin{align*}
(\mathcal{I}^{-\alpha}_{{-}}f )(s)=(\mathcal{D}^\alpha_{-}f )(s) :&=-\frac{d}{ds}\left(\mathcal{I}^{1-\alpha}_{{-}}f \right)(s)\\
&=-\frac{1}{\Gamma(1-\alpha)}\frac{d}{ds}\int_s^{\infty} f(u)(u-s)^{-\alpha}du,\; s\in\R.
\end{align*}
The (right-sided) Riemann--Liouville fractional derivative operator of order $\alpha$ over $[0, T ]$ is
defined by
\begin{gather}\label{equ:diff1}(\mathcal{I}^{-\alpha}_{T_{-}}f )(s)=(\mathcal{D}^\alpha_{T_{-}}f )(s) :=-\frac{d}{ds}(\mathcal{I}^{1-\alpha}_{T_{-}}f )(s), s > 0.\end{gather}
\end{definition}
 If both operators $(\mathcal{D}^\alpha_{T_{-}}f )$ from \eqref{equ:diff1} and $\big(\mathcal{D}_{T-}^{ \alpha}g\big)$ from \eqref{equ:dif} exist for some $\alpha>0$, they coincide. Equality \eqref{equ:dif}  is a Weyl representation of fractional derivative from \eqref{equ:diff1}.
Now, for $H\in(0, 1)$, define weighted fractional integral operators by
$$(\mathbf{K}^H f )(s) := (\mathbf{K}^H_T f )(s) := C(H)s^{\frac12-H}(\mathcal{I}^{H-\frac12}_{T_{-}}(u^{H-\frac12}f (u)))(s),
 s\in (0, T ),$$
 where $C(H)=\left(\frac{2H\Gamma(H+\frac12)\Gamma(\frac32-H)}{\Gamma(2-2H)}\right)^{\frac12}.$

 Let throughout the paper $(\Omega, \mathcal{F},  \mathbb{P} )$ be   a complete probability space   supporting all stochastic processes mentioned below.
 Introduce the fractional Brownian motion $B^H=\{B^H(t), t\geq 0\}$ with Hurst index $H\in(0,1)$ on $(\Omega, \mathcal{F},  \mathbb{P} )$, that is, a Gaussian process with zero mean and covariance function $$R(t,s)=\frac12(s^{2H}+t^{2H}-|t-s|^{2H}).$$ Then, according to \cite{jost,norros}, there exists an one-dimensional Wiener process $W=\{W(t),t\geq 0\}$ on this probability space such that
 \begin{equation}\label{fbmviawin}B^H(t)=\int_0^t(\mathbf{K}^H_t \mathbb{I}_{[0,t]} )(s)dW(s).\end{equation}
 In turn, Wiener process $W=\{W(t), t\in\R^+\}$ is presented via fractional Brownian motion $B^H$ as
 \begin{equation}\label{winviafbm}
 W(t)=(C(H))^{-1}\int_0^ts^{\frac12-H}(\mathcal{I}_{t-}^{\frac12-H} u^{H-\frac12} )(s)dB^H(s).
 \end{equation}

 \subsection{Martingale representation. Elements of Malliavin calculus   and Clark-Ocone representation} Let $W=\{(W_1(t),\ldots,W_m(t)),t\geq 0\}$ be a $m$-dimensional Wiener process. Denote $\mathbb{F}^W=\{\mathcal{F}_t^W, t\geq 0\}$ the filtration generated by $W$ on $(\Omega, \mathcal{F},  \mathbb{P} )$, and let point $T>0$ be fixed. Denote $(\cdot,\cdot)$ and $\|\cdot\|$ the inner product and the Euclidean norm in $\R^m$, correspondingly.
 Consider $\F_T^W$-measurable random variable $\xi$ with $\E\xi^2<\infty$. According to the well-known martingale representation theorem (see, e.g., \cite{ma}), there exists such $m$-dimensional  and $\mathbb{F}^W$-progressively measurable process $\vartheta=\left\{\vartheta(t), \F_t^W, t\in[0,T]\right\}$ that
  \begin{equation}\label{int1}
 \E\int\limits_{0}^{T}\left\|\vartheta(s)\right\|^2ds<\infty
 \end{equation} and
\begin{equation}\label{xi}
\xi=\E(\xi)+\int\limits_{0}^{T}\left(\vartheta(t),dW(t)\right)=\E(\xi)+\sum\limits_{i=1}^{m}\int\limits_{0}^{T}\vartheta_i(t)dW_i(t).
\end{equation}
 According, e.g., to \cite{Karatzas}, representation \eqref{xi} can be generalized to  random variables $\xi$ with $\E(|\xi|)<\infty$  by  replacing \eqref{int1} with property $\int\limits_{0}^{T}\left\|\vartheta(s)\right\|^2ds<\infty$ a.s..

 Now, for any $k\geq 1$ denote by $\mathcal{C}_b^\infty(\R^k)$ the space of bounded infinitely differentiable functions $f:\R^k\rightarrow\R$ with   bounded derivatives of all orders. Let $ \mathcal{S} $ be the class of smooth functionals, i.e.,   random variables of the form
 $$F(\omega)=f(W(t_1,\omega ),\ldots,W(t_n,\omega )),$$
 where $(t_1,\ldots,t_n)\in [0,T]^n$, and the function $f=f(x_{1,1},\ldots,x_{n,m}): R^{nm}\rightarrow\R$,
  $f\in \mathcal{C}_b^\infty(\R^{nm}).$ The stochastic gradient $DF(\omega)$ of the smooth functional $F$ is defined as the $(L_2([0,T])^m$ - valued random variable $DF=\left(D^1F,\ldots,D^m F\right)$ with components
	$$D^iF(\omega)(t)=\sum\limits_{j=1}^{n}\frac{\partial}{\partial x_{i,j}}f\left(W(t_1,\omega),\ldots,W(t_n,\omega)\right)\mathbb{I}_{[0,t_j]}(t), \;\;\; 1\le i \le m.$$
	For each $p\geq 1$  introduce the norm
	$$\left\|F\right\|_{p,1}=\left(\E\left(|F|^p+\left(\sum\limits_{i=1}^{m}\left\|D^i F\right\|^2_{L_2([0,T])}\right)^{\frac{p}{2}}\right)\right)^{\frac{1}{p}}$$
	on $S$, and denote $D_{p,1}$ the Banach space which is the closure of $S$ under $\left\|\cdot\right\|_{p,1}$. It was proved in \cite{Shigekawa} that $DF$ is correctly defined on $D_{p,1}$ by  closure. Given $F\in D_{p,1}$, one can find a measurable process $D_tF(\omega)$ such that for a.e. $\omega\in\Omega$, and a.a. $t\in[0,T]$ $D_t F(\omega)=D F(\omega)(t)$.
	
	The Clark-Ocone representation for any $F\in D_{2,1}$    was introduced in \cite{NP}, and generalized  to $F\in D_{1,1}$ in the  paper \cite{OK}. It has a form
	\begin{equation}\label{F}
	F=\E(F)+\int\limits_{0}^{T}\left(\E\left(D_t F|\F_t^W\right),dW(t)\right)=\E(F)+\sum\limits_{i=1}^{m}\int\limits_{0}^{T}\left(\E D_t^i F|\F_t^W\right)dW_i(t).
	\end{equation}
	Representation \eqref{F} is a clarification of \eqref{xi} in the sense that for $\xi\in D_{2,1}$ and more generally,  for $\xi\in D_{1,1}$, we can specify the form of the process $\vartheta$:
	$$\vartheta(t)=\E\left(D_tF|\F_t^W\right).$$

\section{Notion, examples   and representations  of Wiener-transformable  processes}\label{sec:3}

 \begin{definition}\label{def1} A gaussian process  $G=\{G(t), t\in\R^+\}$ is called $m$-Wiener-transformable if there exists such $m$-dimensional Wiener process $W=\{W(t),\linebreak t\in\R^+\}$ that $G$ and $W$ generate the same filtration, i.e. for any $t\in\R^+$ $$\mathcal{F}_t^G=\mathcal{F}_t^W.$$
 We say that $G$ is $m$-Wiener-transformable to $W$ (evidently, process $W$ can be non-unique.)
 \end{definition}
 \begin{remark}\begin{itemize}
   \item[$(i)$] In the case when $m=1$  we say that the  process $G$ is  Wiener-transformable.
   \item[$(ii)$] Being Gaussian so having moments of any order, $m$-Wiener-transformable process admits at each time $t\in\R^+$
   the martingale    representation $G(t)=\E(G(0))+\sum_{i=1}^m\int_0^tK_i(t,s)dW_i(s),$
   where $K_i(t,s)$ is $\mathcal{F}_s^W$-adapted for any $0\leq s\leq t$ and $\int_0^t\E(K_i(t,s))^2ds<\infty$ for any $t\in\R^+$.
   \end{itemize}
 \end{remark}

   \begin{example}\label{ex1}
Some simple examples of Wiener-transformable processes are:
	\begin{itemize}
	\item[(a)] Geometric Brownian motion involving the Wiener component  and having the form $$S=\left\{S(t)=S(0)\exp\left\{\mu t+\sigma W(t)\right\}, \;\; t\ge 0 \right\},$$
	with $S(0)>0$, $\mu\in\R$, $\sigma>0$, is Wiener-transformable to the underlying Wiener process $W$.
	\item[(b)] Fractional Brownian motion $B^H$ is Wiener-transformable to the Wiener process $W$ with which it is connected via relations \eqref{fbmviawin}--\eqref{winviafbm}.
	\item[(c)] Let $H>\frac{1}{2}$. Then the fractional Ornstein--Uhlenbeck process $Y=\{Y(t),  \linebreak t\ge 0\}$, involving fractional Brownian component and  satisfying the equation
	$$ Y(t)=Y_0+\int_0^t(b-aY(s))ds+\sigma  B^H(t),$$
	where $a,b\in\R$ and $\sigma>0$, is Wiener-transformable to the same Wiener process as the underlying fBm  $B^H$.
	\end{itemize}
	\end{example}
	\begin{example}
	Consider the collection of Hurst indices $\frac{1}{2}\le H_1< H_2<\ldots<H_m<1$ and independent fractional Brownian motions with corresponding Hurst indices $H_i$, $1\le i \le m$. Then the linear combination $\sum\limits_{i=1}^{m}a_iB^{H_i}$ is $m$-Wiener-transformable to the Wiener process $W=(W_1,\ldots,W_m)$, where $W_i$ is such Wiener process  to which fractional Brownian motion $B^{H_i}$ is Wiener-transformable. In particular,  the  mixed fractional Brownian motion  $M^H=W+B^H$, introduced in \cite{Cheridito}, is $2$-Wiener-transformable.
	\end{example}
	Now, let the random variable $\xi$ be $\F_T^W$-measurable, $\E\xi^2<\infty$ and $G=\left\{G(t),\F_t^W, t\in[0,T]\right\}$ be zero mean $W$-transformable Gaussian process. Assume that the covariance function of $G$ satisfies the following two-sided power bounds.
	\begin{itemize}
	\item[$(A)$] There exist $0<H_2\le H_1\le1$ and $C_1, C_2>0$ such that for any $s,t\in [0,T]$
	$$C_1\left|t-s\right|^{2H_1}\le\E\left|G(t)-G(s)\right|^2\le C_2 \left|t-s\right|^{2H_2}.$$
	\end{itemize}
Assume additionally  that the increments of $G$ are positively correlated. More exactly, let the following condition hold
\begin{itemize}
	\item[$(B)$] For any $0 \le s_1 \le t_1 \le s_2 \le t_2\le T$ 	 $$\E\left(G({t_1})-G({s_1})\right)\left(G({t_2})-G({s_2})\right)\ge0.$$

Here are some examples of Gaussian processes satisfying conditions $(A)$ and $(B)$  (for more detail and proofs see, e.g. \cite{mish-shev}):
\begin{itemize}
\item[$(i)$]fractional Brownian motion with index $H\in(0,1)$ satisfies condition $(A)$ with $H_1=H_2=H$ and satisfies condition $(B)$ if $H\in(\frac{1}{2},1)$;
\item[$(ii)$] fractional Ornstein-Uhlenbeck process of the simplified form
\begin{equation*}
Y(t) = Y_0 + a \int_0^t Y(s) ds + B^H(t), \mbox{ } t \geq 0
\end{equation*}
satisfies condition $(A)$ with $(0,1)\ni H=H_1=H_2$. It satisfies condition $(B)$ for $a>0$, otherwise its increments are
 neither positive nor negatively correlated.
\item[$(iii)$] subfractional Brownian motion with index $H$, that is a centered Gaussian process
$G^H=\left\{G^H(t), t \geq 0 \right\}$ with  covariance function
$$\E G^H(t) G^H(s) = t^{2H}+s^{2H} -\frac{1}{2}\left(|t+s|^{2H} + |t-s|^{2H} \right),$$
 satisfies condition $(A)$  with $(0,1)\ni H=H_1=H_2$ and satisfies condition $(B)$ for $H\in(\frac{1}{2},1)$.
\item[$(iv)$] bifractional Brownian motion with indices $H \in (0,1)$ and $K \in (0,1)$, that is a centered Gaussian process with covariance function
 $$ \E G^{H,K}(t) G^{H,K}(s) = \frac{1}{2^K} \left( \left(t^{2H}+s^{2H}\right)^K - |t-s|^{2HK}\right),$$ satisfies condition $(A)$ with $H_1 = H_2 = HK$ and satisfies condition $(B)$ for $HK>\frac{1}{2}$;
\item[$(v)$] Consider Volterra integral transform of Wiener process, that is the process of the form $G(t) = \int_0^t K(t,s) dW(s)$ with non-random kernel $K(t, \cdot) \in L_2[0,t]$ for $t\in[0,T]$. Let the constant $r\in[0,1/2)$ be fixed. Let the following conditions hold.

\begin{itemize}
\item[ $(B1)$] The kernel $K$ is non-negative on $[0,T]^2$  and for any $s\in [0,T]$   $K(\cdot,s)$ is non-decreasing in the first argument;
\item[$(B2)$]  There exist constants   $D_i>0, i=2,3$ and   $1/2<H_2<1$    such that   $$|K(t_2,s) - K(t_1,s)| \leq D_2 |t_2-t_1|^{H_2 }s^{\!-r},\,\,\,s,\; t_1,\;t_2 \in [0,T] $$
      and  $$\ K(t,s)\leq D_3(t-s)^{H_2-1/2}s^{-r};$$
\end{itemize}
      and at least one of the following conditions
\begin{itemize}
    \item[$(B3,a)$]  There exist constants  $D_1>0$ and $H_1\geq H_2 $   such that $$D_1|t_2-t_1|^{H_1}s^{-r}\leq|K(t_2,s) - K(t_1,s)|,\,\,\;\;s,\; t_1,\;t_2 \in [0,T];$$
     \item[$(B3,b)$]There exist constants  $D_1>0$ and $H_1\geq H_2 $   such that $$K(t,s)\geq D_1(t-s)^{H_1-1/2}s^{-r},\,\;\;s,\; t \in [0,T].$$
\end{itemize}
Then the Gaussian  process  $G(t) = \int_0^t K(t,s) dW(s)$, satisfies  conditions $(A)$  on any subinterval $[1-\delta, 1]$ for $0<\delta<1$  with powers $H_1, H_2$, and condition   $(B)$.

\end{itemize}
\item[$(vi)$] As before, let us take the collection of Hurst indices $\frac{1}{2}\le \widetilde{H}_1< \widetilde{H}_2<\ldots<\widetilde{H}_m<1$ and independent fractional Brownian motions with corresponding Hurst indices $\widetilde{H}_i$, $1\le i \le m$. Then the linear combination $\sum\limits_{i=1}^{m}a_iB^{\widetilde{H}_i}$ satisfies condition  $(A)$   with $H_1=H_2=\widetilde{H}_1$, and condition $(B)$.
\end{itemize}
The next result is proved in \cite{mish-shev}.
\begin{theorem}\label{thm1} (Representation theorem)
Assume that an adapted Gaussian process $G=\{G(t)$, $t\in[0,T]\}$ satisfies condition  $(A)$  with  $0<2H_1-1<H_2<H_1$ and condition $(B)$. Additionally let random variable $\xi$ satisfy the following condition
\begin{itemize}
	\item[$(C)$] $\xi=U(T)$ for some adapted process $U\in C^{\rho}[0,T]$ with $\rho>\rho_0$, where
	$\rho_0=\frac{(1+H_2)(H_1-H_2)}{H_2+1-2H_1}.$
\end{itemize}
Then there exists an adapted process $\psi$ that $\left\|\psi\right\|_{\alpha,T}<\infty$ for some $\alpha\in \left(1-H_2,\frac{1}{2}\right)$ and $\xi$ admits the representation
 \begin{equation}\label{reprez}
\xi=\int\limits_{0}^{T}\psi(s) dG(s),
\end{equation}  almost surely, where $\int\limits_{0}^{T}\psi(s) dG(s)$ is understood as a generalized Lebesgue-Stieltjes integral.
\end{theorem}

 \begin{remark} As it was mentioned in \cite{mish-shev}, it is sufficient to require the properties $(A)$ and $(B)$ to hold on some subinterval $[1-\delta,1]$. Also, in the case $H_1=H_2$ we have $\rho_0=0$, so, we can consider any  $\rho>0$ in condition $(C)$. Therefore, the representation theorem holds for fBm with $H>\frac{1}{2}$, fractional Ornstein-Uhlenbeck process with positive drift and $H>\frac{1}{2}$, subfractional and bifractional Brownian motions with $H>\frac{1}{2}$ and $HK>\frac{1}{2}$ correspondingly, and the linear combination of fBms with $H_i>\frac{1}{2}$, all of the above under additional condition $(C)$ with any $\rho>0$. Moreover, the representation theorem holds for
Volterra integral transform of a Wiener process under additional assumption $0 < 2H_1-1<H_2<H_1$ and condition $(C)$.
 As it was mentioned in \cite{mish-shev}, the representation theorem is valid for a fractional Ornstein-Uhlenbeck
 process with negative drift coefficient too. Indeed, we can annihilate the drift of the fractional Ornstein-Uhlenbeck process with the help of Girsanov theorem,
   transforming a fractional  Ornstein-Uhlenbeck process with negative drift to a fractional Brownian motion $\widetilde{B}^H$.
    Then, assuming  condition $(C)$, we represent the random variable $\xi$ as
    $\xi=\int_0^T\psi(s)d\widetilde{B}^H(s)$ on the new probability space. Finally, we return to the original probability space.

For the case of the mixed model $W +B ^H$, $H>\frac{1}{2}$,  according to \cite{shev-viita}, if we consider the natural filtration generated by $(W,B ^H)$ and  a random variable $\xi$ with  $\E\xi^2<\infty$, then the representation $\xi= \int_0^T\psi(s)d\left(W(s) + B^H(s) \right)$ holds, where $\int_0^T \psi(s) dW(s)$ is an It\^{o} integral, and $\int_0^T \psi(s) dB^H(s)$ is a generalized Lebesgue-Stieltjes integral. We do not need any ``H\"{o}lder'' condition on $\xi$ in this case.
 \end{remark}

Consider now conditions supplying H\"{o}lder properties of stochastic integrals and their quadratic characteristics.
\begin{lemma}\label{lem1}
Let $\vartheta=\{\vartheta(t), t\in [0,T]\}$ be real-valued progressively measurable process for which $$ \Prob\left\{\int\limits_{0}^{T}\vartheta^2(s)ds<\infty\right\}=1$$ and one of the following conditions hold:
\begin{itemize}
\item[$(i)$] $\Prob\left\{\sup\limits_{0\le t \le T}\left|\vartheta(t)\right|\ge C\right\}\rightarrow 0$ as $C\rightarrow\infty$;
\item[$(ii)$] There exists such $\delta>0$ that $\sup\limits_{t\in[0,T]}\E\left|\vartheta(t)\right|^{2+\delta}<\infty$;
\item[$(iii)$] There exists such $\delta>0$ that $\E\int\limits_{0}^{T}|\vartheta(t)|^{4+\delta}dt<\infty$.
\end{itemize}
Then under condition $(i)$  stochastic integral $\int\limits_{0}^{t}\vartheta(s)dW(s)$ is H\"{o}lder up to order $\frac{1}{2}$ and its quadratic characteristic $\int\limits_{0}^{t}\vartheta^2(s)ds$ is Lipschitz; under condition $(ii)$  stochastic integral $\int\limits_{0}^{t}\vartheta(s)dW(s)$ is H\"{o}lder of  order   $\frac{\delta}{4+2\delta}$ and its quadratic characteristic $\int\limits_{0}^{t}\vartheta^2(s)ds$ is H\"{o}lder of  order   $\frac{\delta}{2+\delta}$, and under condition $(iii)$ stochastic integral $\int\limits_{0}^{t}\vartheta(s)dW(s)$ is H\"{o}lder of  order   $\frac{\delta}{8+2\delta}$ and its quadratic characteristic $\int\limits_{0}^{t}\vartheta^2(s)ds$ is H\"{o}lder of  order   $\frac{\delta}{4+\delta}$.
\end{lemma}
\begin{remark} Note that none of the conditions $(i)$--$(iii)$ can   be embedded   into another one. Indeed,
function $\vartheta(t)=t^{-\frac{1}{4+2\delta}}$ satisfies condition $(iii)$ but not $(i)$ and $(ii)$. Further, let for the technical simplicity, $T=1$. Then the process $\vartheta(t)=\exp\left\{\frac12 {W^2(t)}\right\}$ satisfies condition $(i)$ but not $(ii)$ and $(iii)$. Indeed, for $C>1$
\begin{align*}
&\Prob\left\{\sup\limits_{0\le t \le T}\exp\left\{\frac12 {W^2(t)}\right\}\ge C\right\}\\
&=\Prob\left\{\sup\limits_{0\le t \le T}|W(t)|\ge \sqrt{2\log C }\right\}\rightarrow 0, \;\; C\rightarrow\infty,
\end{align*} so, $(i)$ holds and, in addition,
$$\E\int\limits_{0}^{T}|\vartheta(t)|^2dt=\int\limits_0^1\int\limits_\R e^{\frac{x^2t}{2}}\frac{e^{-\frac{x^2}{2}}}{\sqrt{2\pi}}dxdt=\frac{1}{\sqrt{2\pi}}\int\limits_0^1\int\limits_\R e^{\frac{x^2(t-1)}{2}}dxdt$$
$$=\frac{1}{\sqrt{2\pi}}\int\limits_0^1\int\limits_\R e^{\frac{-x^2z}{2}}dxdz=\int\limits_0^1z^{-\frac{1}{2}}dz<\infty.$$
However, for any $\delta>0$ and any $t>0$
$$\E|\vartheta(t)|^{2+\delta} =\frac{1}{\sqrt{2\pi}}\int\limits_\R \exp{\left\{\frac{(2+\delta)x^{2}t }{2}\right\}}e^{-\frac{x^2}{2}}dx =\infty $$ for $t\in[\frac{1}{2+\delta},1] $ which obviously means that neither $(ii)$  nor $(iii)   hold. $ Furthermore, the process $\vartheta(t)=\left(\frac{e^{\frac{W^2(t)}{3}}-1}{t}\right)^{\frac{1}{2+\delta}} $ satisfies condition $(ii)$ but not $(i)$ and $(iii)$. Indeed, it follows from the log log law that $$ \sup\limits_{0\le t \le 1} \frac{e^{\frac{W^2(t)}{3}}-1}{t} \geq \sup\limits_{0\le t \le 1}   \frac{ W^2(t)}{3 t}   =\infty,$$
so $(i)$ does not hold, while

\begin{eqnarray}\sup_{0\leq t\leq 1}\E|X_t|^{2+\delta}& = &\sup_{0\leq t\leq 1} \E \frac{e^\frac{W^2(t)}{3}-1}{t} \\
&=& \nonumber
\frac{1}{\sqrt{2\pi}}\sup_{0\leq t \leq 1}\int_{\mathbb{R}}\left(\frac{ e^{\frac{x^2t}{3}}-1}{t}\right)e^{-\frac{x^2}{2}}dx\\
&=& \nonumber\frac{1}{\sqrt{2\pi}}\int_{\mathbb{R}}\left(  e^{\frac{x^2}{3}} -1\right)e^{-\frac{x^2}{2}}dx<\infty,
\end{eqnarray}
because for any $x\in\R$  function $f(x,t)=\frac{e^{\frac{x^2t}{3}}-1}{t}$ increases in $t$ and achieves its maximal value on the interval $(0,1]$ at $t=1$. Finally,
$\E|\vartheta(t)|^{4+\delta} \linebreak= t^{-\frac{4+\delta}{2+\delta}} \E\left( e^{\frac{W^2(t)}{3}}-1\right)^{\frac{4+\delta}{2+\delta}}=+\infty$ for any $0<\delta<2$ and any $\frac{6+3\delta}{8+2\delta}<t<1$, therefore, $(iii)$ does not hold.

\end{remark}

\begin{proof}
Let condition $(i)$ hold. Consider the stochastic processes $$Z^C(t)=\int\limits_0^t\left(\left(\vartheta(u)\wedge C\right)\vee (-C)\right)dW(u) \;\;\; \text{and} \;\;\; Z(t)=\int\limits_0^t\vartheta(u)dW(u).$$
Then, according to Burkholder inequalities,  for any $0<s<t\le T$ and  for any $p>1$
$$\E\left|Z^C(t)-Z^C(s)\right|^p\le C_p \E\left(\int\limits_s^t|(\vartheta(u)\wedge C)\vee (-C)|^2du\right)^{\frac{p}{2}}\le C_p C^p (t-s)^{\frac{p}{2}}.$$
Recall  the well-known Kolmogorov theorem which states that under the condition
$$\E\left|U(t)-U(s)\right|^p\le C    (t-s)^{1+q}$$
the trajectories of stochastic process $U$ are H\"{o}lder with probability 1 of order $\frac{q}{p}$. Applying this result, we conclude that  process $Z^C$ is H\"{o}lder with any order $\frac{1}{2}-\frac{1}{p}$, thus it is H\"{o}lder  up to order $\frac{1}{2}$. But
$$\Prob\left\{Z(t)\not=Z^C(t)\right\}=\Prob\left\{\sup\limits_{0\le t \le T}|\vartheta(t)|\ge C\right\}\rightarrow 0, \;\;\; C\rightarrow\infty.$$
Therefore, a.a. trajectories of $Z$ are H\"{o}lder up to order $\frac{1}{2}$.

Let condition $(ii)$ hold. Then for any $0\le s < t \le T$ and according to Lemma 4.12 \cite{Lip_Sh} for any $\delta>0$
\begin{equation}
\label{eq1}
\begin{gathered}
\E\left|\int\limits_s^t \vartheta(u)dW(u)\right|^{2+\delta}\le C_\delta(t-s)^\frac{\delta}{2}\E\int\limits_s^t\left|\vartheta(u)\right|^{2+\delta}du\\
\le\sup\limits_{0\le t \le T}\E|\vartheta(t)|^{2+\delta}(t-s)^{1+\frac{\delta}{2}}.
\end{gathered}
\end{equation}
It means that now the  process $Z$ is H\"{o}lder of order $\frac{\delta}{4+2\delta}.$
Let condition $(iii)$ hold. Then, similarly to \eqref{eq1}, for any $\delta>0$
$$\E \left|\int\limits_s^t\vartheta(u)dW(u)\right|^{4+\delta}\le C_\delta(t-s)^{1+\frac{\delta}{2}}\E\int_s^t\left|\vartheta(u)\right|^{4+\delta}du$$
and  we conclude similarly to $(ii)$ that the process  $Z$  is H\"{o}lder of order $\frac{\delta}{8+2\delta}.$
Concerning quadratic characteristic $\langle Z\rangle(t) = \int_0^t \vartheta^2(s)ds$ , under condition $(i)$ we have  the process
$\int_0^t \left( | \vartheta(s)|\wedge C\right)^2 ds$ is Lipschitz, but  $$\Prob \left\{\int_0^t \left(|\vartheta(s)| \right)^2ds \neq \int_0^t \vartheta^2(s)ds, 0 \leq t \leq T \right\}\rightarrow 0, C\rightarrow \infty.$$ Therefore, $\int_0^t \vartheta^2(s) ds$ is Lipschitz.
Under condition $(ii)$
\begin{eqnarray}
\E\left(\int_s^t \vartheta^2(u)du \right)^{1 + \frac{\delta}{2}}&\leq& \E \left(\int_s^t |\vartheta(u)|^{2+\delta}du\right)(t-s)^\frac{\delta}{2}
\\
&\leq& \sup_{0 \leq t \leq T} \E|\vartheta(t)|^{2+\delta}(t-s)^{1+\frac{\delta}{2}}.
\end{eqnarray}
Therefore, $\int_0^t \vartheta^2(s)ds$ is H\"{o}lder of order $\frac{\delta}{2+\delta}.$ Under condition $(iii)$
\begin{equation}
\begin{gathered}
\E\left(\int_s^t \vartheta^2(u)du \right)^{2 + \frac{\delta}{2}} \leq  \E \left(\int_s^t |\vartheta(u)|^{4+\delta}du\right)(t-s)^{1+\frac{\delta}{2}}\\
\leq \E\int_0^T|\vartheta(u)|^{4+\delta}du(t-s)^{1+\frac{\delta}{2}},
\end{gathered}
\end{equation}
and, consequently, $\int_0^t \vartheta^2(u)du$ is H\"{o}lder of order $\frac{\delta}{4+ \delta}$.
\end{proof}
\begin{remark}
 For $\xi\in D_{1,1}$ we can rewrite conditions $(i)$--$(iii)$ replacing $\vartheta(t)$ with $\E(D_t\xi|\F_t)$.
  \end{remark}

\section{Expected utility maximization for unrestricted capital profiles}\label{sec:4}
Consider the problem of maximizing the expected utility. Our goal is to characterize the optimal asset profiles in the framework of the markets with  risky assets involving  Gaussian processes satisfying conditions of Theorem \ref{thm1}. We follow the general approach described in \cite{ekel} and \cite{karat}, but apply its  interpretation from \cite{Foll-Sch}.  We fix $T>0$ and from now on consider $\F_T^W$-measurable random variables. Let the utility function $u:\R\rightarrow\R$ be strictly increasing and strictly concave, $L^0(\Omega, \F_T^W, \Prob)$ be the set of all $\F_T^W$-measurable random variables, and let the set of admissible capital profiles coincides with $L^0(\Omega, \F_T^W, \Prob)$. Let $\Prob^*$ be a probability measure on $(\Omega, \F_T^W)$, which is equivalent to $\Prob$, and denote $\varphi(T)=\frac{d\Prob^*}{d\Prob}$.  The budget constraint is given  by $\E_{\Prob^*}(X)=w$, where $w>0$ is some number that can be in some cases, but not obligatory, interpreted as the initial wealth. Thus the budget set is defined as
$$\mathcal{B}=\left\{X\in L^0\left(\Omega, \F_T^{W}, \Prob\right)\cap L^1\left(\Omega, \F_T^W, \Prob^* \right)|\E_{\Prob^*}(X)=w\right\}.$$
The problem is to find such $X^*\in\mathcal{B}$, for which $\E( u(X^*))=\max\limits_{X\in\mathcal{B}} \E( u(X))$. Consider the inverse function $I(x)=(u'(x))^{-1}$.
 \begin{theorem}(\cite{Foll-Sch}, Theorem 3.34)\label{Theorem main for max}
Let the following condition hold:
\label{Follmer-Sch}
   Strictly increasing and strictly concave utility function $u:\R\rightarrow\R$ is continuously differentiable, bounded from above and $$\lim_{x\downarrow -\infty} u'(x)=+\infty.$$
Then the solution of this maximization problem has a form $$X^*=I(c\varphi(T)),$$   under additional assumption  that $\E_{\Prob^*}(X^*)=w$.
\end{theorem}
To connect the solution of maximization problem with specific $W$-transform\-able Gaussian process describing the price process, we consider the following items.

1. Consider  random variable $\varphi(T)$,  $\varphi(T)>0$ a.s. and let $\E(\varphi(T))=1.$  Being the terminal value of a positive martingale $\varphi=\{\varphi_t=\E(\varphi(T)|\F_t^W), t\in[0,T]\}$, $\varphi(T)$ admits the following representation
\begin{equation}\label{fi}
\varphi(T)=\exp\left\{\int\limits_0^T \vartheta(s)dW_s-\frac{1}{2}\int\limits_0^T \vartheta^2(s)ds\right\},
\end{equation}
  where $\vartheta$ is a real-valued progressively measurable process for which $$ \Prob\left\{\int\limits_{0}^{T}\vartheta^2(s)ds<\infty\right\}=1.$$
Assume that process $X$ satisfies one of the conditions $(i)$--$(iii)$ of Lemma \ref{lem1}. Then $\varphi(T)$ is the terminal value of a H\"{o}lder process of the order specified by Lemma \ref{lem1}.

2. Consider $W$-transformable Gaussian process $G=\{G(t), t\in[0,T]\}$ satisfying conditions $(A)$ with some  $0<2H_1-1<H_2<H_1$ and $(B)$, and introduce the set
\begin{gather*}
\mathcal{B}_w^G=\bigg\{\psi=\left\{\psi(t),\F_t^W, t\in[0,T]\right\}:\,\,\,\,\,\,\\\mbox{there exists a generalized Lebesgue-Stieltjes integral} \\
 \int\limits_0^t \psi(s)dG(s),\;\; t\in[0,T] \;\; \text{and} \;\; \E\bigg(\varphi(T)\int\limits_0^T\psi(s) dG(s)\bigg)=w\bigg\}.
\end{gather*}
Denote $H_3=\frac{(1+H_2)(H_1-H_2)}{H_1+1-2H_2}$.
\begin{theorem} \label{TheoremIntRepresentation} Let the following conditions hold
\begin{itemize}
 \item[$(i)$] Function $I(x),x\in\R$ is H\"older of order $\lambda>0$.
 \item[$(ii)$] Stochastic process $\vartheta$ in representation \eqref{fi} satisfies one of assumptions $(i)$--$(iii)$ of Lemma \ref{lem1}.
     \item[$(iii)$] Gaussian process $G$ satisfies condition $(A)$ with $0<2H_1-1<H_2<H_1$ and condition $(B)$.
   \item[$(iv)$]    $\frac{\lambda}{2}>H_3$ in the case when $\vartheta$ satisfies assumption  $(i)$, $\frac{\lambda\delta}{4+2\delta}>H_3$ in the case when $\vartheta$ satisfies assumption  $(ii)$  and $\frac{\lambda\delta}{8+2\delta}>H_3$ in the case when $\vartheta$ satisfies assumption  $(iii)$ of Lemma \ref{lem1}.
       \item[$(v)$] There exists such $c\in\R$ that $\E(\varphi(T)I(c\varphi(T)))=w$.
\end{itemize}        Then the random variable $X^*=I(c\varphi(T))$  admits the representation
\begin{equation}\label{reprmain}
X^*=\int\limits_0^T\overline{\psi}(s)dG(s),
\end{equation}
with some $\overline{\psi}\in \mathcal{B}_w^G,$
and
\begin{equation}\label{maxim}\E( u(X^*))=\max\limits_{\psi\in\mathcal{B}_w^G}\E \left(u\left(\int\limits_0^T\psi(s)dG(s)\right)\right).\end{equation}
\end{theorem}
\begin{proof}
Conditions $(i)$, $(iii)$ and $(iv)$ of Theorem \ref{TheoremIntRepresentation} together with Lemma \ref{lem1} supply that for any $c \in \mathbb{R}$ the random variable $\xi= I(c \varphi(T))$   is the final value of a H\"{o}lder process
$$
U(t)= I(c \varphi(t))  = I\left(c \exp\left\{\int_0^t \vartheta(s) d W(s) - \frac{1}{2} \int_0^t \vartheta^2(s) ds\right\}\right),
$$
and the H\"{o}lder index, being $\frac{\lambda}{2}$, $\frac{\lambda \delta}{4 + 2 \delta}$, or $\frac{\lambda \delta}{8 + 2 \delta}$, exceeds $H_3$. It means that condition $(C)$ of Theorem \ref{thm1} holds. Other conditions of Theorem \ref{thm1} are supplied by condition $(ii)$. Therefore, representation (\ref{reprmain}) follows from Theorem \ref{thm1}. Now, assume that (\ref{maxim}) is not valid, and there exists
$\psi_0 \in \mathcal{B}_w^G$ such that $\E\left(\varphi(T)\int_0^T \psi_0(s) d G(s)\right)=w$, and $\E u \left( \int_0^T \psi_0(s) d G(s) \right)>\E u(X^*)$. But in this case $\int_0^T \psi_0(s) d G(s)$ belongs to $\mathcal{B}$, and we get a contradiction with Theorem \ref{Theorem main for max}.
\end{proof}

\begin{example}
Let $u(x) = 1 - e^{- \beta x}$ be an exponential utility function with constant absolute risk aversion $\beta>0$. In this case
$I(x) = - \frac{1}{\beta} \log ( \frac{x}{\beta})$. Assume that
$$\varphi(T) = \exp \left\{ \int_0^T \vartheta(s) dW(s) - \frac{1}{2} \int_0^T\vartheta^2(s) ds \right\}$$ is chosen in such a way that
\begin{equation}\begin{gathered}
\label{eq1ex41}
\E \left( \varphi(T) |\log \varphi(T)|\right)\\ =\E \bigg( \exp\left\{ \int_0^T \vartheta(s) dW(s)  - \frac{1}{2} \int_0^T\vartheta^2(s) ds \right\}\\ \times
\left|\int_0^T \vartheta(s) dW(s) - \frac{1}{2} \int_0^T\vartheta^2(s) ds\right|
\bigg)<\infty.
\end{gathered}\end{equation}
Then, according to Example 3.35 from \cite{Foll-Sch}, the optimal profile can be written as
\begin{equation}
\label{ex41optprofile}
X^* = - \frac{1}{\beta} \left( \int_0^T \vartheta(s) dW(s) - \frac{1}{2} \int_0^T\vartheta^2(s) ds \right) + w + \frac{1}{\beta} H(\Prob^*|\Prob),
\end{equation}
where $H(\Prob^*|\Prob) = \E \left(\varphi(T) \log \varphi(T)\right)$, condition (\ref{eq1ex41}) supplies that $H(\Prob^*|\Prob)$ exists, and the maximal value of the expected utility is
$$
\E (u(X^*)) = 1 - \exp\left\{-\beta w - H(\Prob^*|\Prob) \right\}.
$$
Let $\varphi(T)$ be chosen in such a way that the corresponding process $\vartheta$ satisfies one of the conditions $(i)$--$(iii)$ of Lemma \ref{lem1}. Also, let $W$-transformable process $G$ satisfy conditions $(A)$ and $(B)$ of Theorem \ref{Follmer-Sch}, and $H_3<\frac{1}{2}$ in the case when condition $(i)$ holds, $H_3 <\frac{\delta}{4 + 2 \delta}$ in the case when condition $(ii)$ holds, and $H_3 < \frac{\delta}{8 + 2 \delta}$ in the case when condition $(iii)$ holds.
 Then  we can conclude directly from representation (\ref{ex41optprofile}) that conditions of Theorem \ref{Follmer-Sch} hold. Therefore, the optimal profile $X^*$  admits the representation  $X^* =   \int_0^T \psi (s) d G(s).$
\end{example}
 \begin{remark} Similarly, under the same conditions as above, we can conclude  that  for any constant $d\in \mathbb{R}$ there exists $\psi_d$ such that $X^* = d + \int_0^T \psi_d(s) d G(s).$ Therefore, we can start from any initial value of the capital and achieve the desirable wealth. In this sense, $w$ is not necessarily the initial wealth as it is often assumed in the semimartingale framework, but  is rather a budget constraint in the generalized sense.
 \end{remark}
 \begin{remark}
 In the case when $W$-transformable Gaussian process $G$ is a semimartingale (one of the simplest possibilities  is presented in Example \ref{ex1}), we can use Girsanov's theorem in order to get the representation, similar to \eqref{reprmain}. Indeed,
  let, for example, $G$ be a Gaussian process of the form $G(t)=\int_0^t\mu(s)ds +\int\limits_0^t a(s)dW(s)$, $|\mu(s)|\leq \mu$, $a(s)>a>0$ are non-random measurable functions, and $\xi$ is $\F_T^W$-measurable random variable, $\E(\xi^2)<\infty$. Then   we transform $G$ into $\widetilde{G}=\int\limits_0^\cdot a(s)d\widetilde{W}(s)$, with the help of equivalent probability measure $\widetilde{\Prob}$ having Radon--Nikodym derivative $$\frac{d\widetilde{\Prob}}{d\Prob}=\exp\left\{-\int_0^T\frac{\mu(s)}{a(s)}dW(s)-\frac{1}{2}
\int_0^T\left(\frac{\mu(s)}{a(s)}\right)^2d s \right\}.$$ With respect to this measure $\E_{\widetilde{\Prob}}|X^*|<\infty$, and we get the following  representation
\begin{gather}\label{repraux}
X^*=\E_{\widetilde{\Prob}}(X^*)+\int\limits_0^T\psi(s)d\widetilde{W}_s=\E_{\widetilde{\Prob}}(X^*)+\int\limits_0^T
\frac{\psi(s)}{a(s)}d\widetilde{G}(s)\\=\E_{\widetilde{\Prob}}(X^*)+\int\limits_0^T
\frac{\psi(s)}{a(s)}d{G}(s)
=\E_{\widetilde{\Prob}}(X^*)+\int\limits_0^T\psi(s) {\mu(s)}  ds+\int\limits_0^T
 {\psi(s)} dW(s).
\end{gather}
Representations \eqref{reprmain} and \eqref{repraux} have the following distinction: \eqref{reprmain} ``starts'' from $0$ (but can start from any other constant) while \eqref{repraux} ``starts'' exactly from $\E_{\widetilde{\Prob}}(X^*)$.
\end{remark}
\begin{remark}
As we can see, the solution of the utility maximization problem for $W$--transformable process depends on the process in  indirect way, through the random variable $\varphi(T)$ such that $\E\varphi(T)=1$, $\varphi(T)>0$ a.s.. Also, this solution depends on whether or not we can choose the appropriate value of $c$, but this is more or less a technical issue. Let us return to the choice of $\varphi(T)$. In the case of the semimartingale market, $\varphi(T)$ can be reasonably chosen as the likelihood ratio of some martingale measure, and the choice is unique in the case of the complete market. The non-semimartingale market can contain some hidden semimartingale structure. To illustrate this, consider two examples.
\end{remark}
\begin{example}\label{ex4.2}
Let the market consist of bond $B$ and stock $S$, $$B(t)=e^{rt},\;S(t)=\exp\left\{\mu t +\sigma B_t^{H}\right\},$$ $r\geq0$, $\mu\in\R$, $\sigma>0$, $H>\frac{1}{2}$. The discounted price process has a form $\vartheta(t)=\exp\left\{(\mu-r)t+\sigma B_t^H\right\}$. It is well-known that such market admits an arbitrage, but even in these circumstances the utility maximization problem makes sense.
Well, how to choose $\varphi(T)$? There are at least two natural approaches.
\begin{itemize}
 \item[{1.}] Note  that for $H>\frac{1}{2}$ the kernel $K_t^H$ from (\ref{fbmviawin}) has a form
$$
\left(K_t^H \mathbb{I}_{[0,t]} \right)(s)= C(H) s^{\frac{1}{2}- H}\int_s^t u^{H-\frac{1}{2}}(u-s)^{H-\frac{3}{2}} du,
$$
and representation (\ref{winviafbm}) has a form
$$
W(t) = \left(C(H)\right)^{-1}\int_0^t s^{\frac{1}{2}-H} K^*(t,s) d B_s^H,
$$
where
\begin{equation*}\begin{gathered}
K^*(t,s) = \Big(t^{H-\frac{1}{2}}(t-s)^{\frac{1}{2} - H} - \left(H-\frac{1}{2}\right)\int_s^t u^{H-\frac{3}{2}}(u\\-s)^{\frac{1}{2}-H}du\Big)\frac{1}{\Gamma\left(\frac{3}{2} - H\right)}.
\end{gathered}\end{equation*}

Therefore,
\begin{eqnarray*}
& &\left(C(H)\right)^{-1}\int_0^t s^{\frac{1}{2}-H} K^*(t,s) d \left( (\mu - r) s + \sigma B_s^H\right)\\
&=& \sigma W(t) + \frac{\mu-r}{C(H)} \int_0^t s^{\frac{1}{2}-H} K^*(t,s) ds\\
&=& \sigma W(t) + \frac{\mu-r}{C(H)\Gamma\left(\frac{3}{2} - H\right)} \int_0^t\left( s^{\frac{1}{2}-H} t^{H- \frac{1}{2}} (t-s)^{\frac{1}{2}-H}\right.\\
& &{}- \left.\left(H-\frac{1}{2}\right)s^{\frac{1}{2}-H} \int_s^t u^{H-\frac{3}{2}}(u-s)^{\frac{1}{2}-H}du \right)ds\\
&=&\sigma W_t + \frac{\mu -r}{C(H) \Gamma(\frac{3}{2}-H)}\frac{\Gamma^2(\frac{3}{2}-H)}{(\frac{3}{2}-H)\Gamma (2 - 2 H)} t^{\frac{3}{2}-H}\\
&=& \sigma W_t + (\mu-r) C_1(H) t^{\frac{3}{2}-H},
\end{eqnarray*}
where
$$
C_1(H) = \left(\frac{3}{2}-H\right)^{-1} \left(\frac{\Gamma(\frac{3}{2} - H)}{2H\Gamma(2 - 2H)\Gamma(H+\frac{1}{2})} \right)^\frac{1}{2}.
$$
In this sense we say that the model involves a hidden semimartingale structure.\\
Consider a virtual semimartingale asset
\begin{align*}
\hat{Y}(t) &= \exp \left\{ (C(H))^{-1} \int_0^t s^{\frac{1}{2}-H}K^*(t,s) d \log Y(s) \right\} \\
&= \exp\left\{\sigma W_t + (\mu-r) C(H)t^{\frac{3}{2}-H}\right\}.
\end{align*}
We see that measure $\Prob^*$ such that
\begin{equation}\begin{split}\label{ex42ChangeMeasure}
\frac{d\Prob^*}{d\Prob} &=\exp \left\{ - \int_0^T \left(\frac{(\mu-r)c_2(H)}{\sigma} s^{\frac{1}{2}-H}+ \frac{\sigma}{2}\right)dW_s\right. \\
&  \quad \left.- \frac{1}{2} \int_0^T \left(\frac{(\mu-r)c_2(H)}{\sigma} s^{\frac{1}{2}-H}+ \frac{\sigma}{2}\right)^2 ds\right\},
\end{split}\end{equation}
where $C_2(H) = C_1(H) \left(\frac{3}{2}-H\right),$ reduces $\hat{Y}(t)$ to the martingale of the form $\exp \left\{\sigma W_t - \frac{\sigma^2}{2} t\right\}$.
Therefore, we can put $\varphi(T) = \frac{d\Prob^*}{d\Prob}$ from (\ref{ex42ChangeMeasure}). Regarding the H\"{o}lder property, $\vartheta(s)= s^{\frac{1}{2}-H}$ satisfies condition $(iii)$ of Lemma \ref{lem1} for $H<\frac{3}{4}$. Therefore, for $H<\frac{3}{4}$ and for utility function $u(x) = 1 - e^{-\alpha x}$ we have
$$
X^* = \frac{1}{\alpha} \left(\int_0^T \varsigma(s) dW_s - \frac{1}{2} \int_0^T \varsigma_s^2 ds \right) + W + \frac{1}{2} H(\Prob^*|\Prob),
$$
where $\varsigma(s) = \frac{(\mu-r)C_2(H)}{\sigma} s^{\frac{1}{2}-H}+ \frac{\sigma}{2}$, and $|H(\Prob^*|\Prob)|<\infty.$
\item[{2.}] It was proved in \cite{Dung} that the fractional Brownian motion $B^H$ is the limit in $L_p(\Omega,\mathcal{F},\Prob)$ for any $p>0$ of the process
$$B^{H,\epsilon}(t)=\int\limits_0^tK(s+\epsilon,s)dW(s)+\int\limits_0^t\psi_\epsilon (s)ds,$$
where $W$ is he underlying Wiener process, i.e. $B^{H}(t)=\int\limits_0^tK(t,s)dW(s),$ where
\begin{gather*}
K(t,s)=C_H s^{\frac{1}{2}-H}\int\limits_s^t u^{H-\frac{1}{2}}(u-s)^{H-\frac{3}{2}}du,\\
\psi_\epsilon(s)=\int\limits_0^s\partial_1K(s+\epsilon,u)dW_u,\\
\partial_1K(t,s)=\frac{\partial K(t,s)}{\partial t}=C_Hs^{\frac{1}{2}-H}t^{H-\frac{1}{2}}(t-s)^{H-\frac{3}{2}}.
\end{gather*}
Consider prelimit market with discounted risky asset price $Y^{\epsilon}$ of the form
$$Y^{\epsilon}(t)=\exp{\left\{(\mu-r)t+\sigma\int\limits_{0}^{t}\psi_\epsilon(s)ds+\sigma\int\limits_0^tK(s+\epsilon,s)dW_s\right\}}.$$
This financial market is arbitrage-free and complete, and the unique martingale measure has the Radon-Nikodym derivative
$$\varphi_\epsilon(T)=\exp\left\{-\int\limits_0^T\zeta_\epsilon(t)dW_t-\frac{1}{2}\int\limits_0^T\zeta^2_\epsilon(t)dt\right\},$$
where
$$\zeta_\epsilon(t)=\frac{\mu-r+\sigma\psi_\epsilon(t)}{\sigma K(t+\epsilon,t)}+\frac{1}{2}\sigma K(t+\epsilon,t).$$
Note that $K(t+\epsilon,t)\rightarrow 0$ as $\epsilon\rightarrow0$. Furthermore, $\rho_t=\frac{\mu-r+\sigma\psi_\epsilon(t)}{\sigma K(t+\epsilon,t)}$ is a Gaussian process with $\E\rho_t=0$ and
\begin{gather*}
\var \zeta_\varepsilon(t)=\int\limits_{0}^{t}\left(\frac{\partial_1 K(t+\epsilon,u)}{ K(t+\epsilon,t)}\right)^2du\\
=\int\limits_{0}^{t}\left(\frac{u^{1/2-H}(t+\epsilon)^{H-1/2}(t+\epsilon-u)^{H-3/2}}{t^{1/2-H}\int\limits_t^{t+\epsilon}v^{H-1/2}(v-t)^{H-3/2}}\right)^2du\\
\ge\epsilon^{1-2H}\int\limits_0^t(t+\epsilon-u)^{2H-3}du=\frac{\epsilon^{1-2H}t}{2-2H}\left(\epsilon^{2H-2}-(t+\epsilon)^{2H-2}\right)\rightarrow\infty.
\end{gather*}
Therefore, we can not get a reasonable limit of $\varphi_\epsilon(T)$ as $\epsilon\rightarrow0.$ Thus one should use this approach  with great caution.
\end{itemize}
\end{example}

\section{Expected utility maximization for restricted capital profiles}\label{sec:5}
Consider now the case when the utility function $u$ is defined on some interval $(a,\infty)$. Assume for technical simplicity that $a=0$. Therefore, in this case case $\mathcal{B}_0$ of admissible capital profiles has a form
$$\mathcal{B}_0=\left\{X\in L^0(\Omega,\F,\Prob):X\ge0 \;\; \text{a.s. and} \;\; \E(\varphi(T) X)=w\right\}.$$
Assume that the utility function $u$ is continuously differentiable on $(0,\infty)$, introduce $\pi_1=\lim\limits_{x\uparrow\infty}u'(x)\ge0$, $\pi_2=u'(0+)=\lim\limits_{x\downarrow\infty}u'(x)\le +\infty$, and define $I^+:(\pi_1,\pi_2)\longrightarrow(0,\infty)$
as the continuous, bijective and strictly decreasing inverse functions of $u'$ on $(\pi_1,\pi_2)$.

Extend $I^+$ to the full half axis $\left[0,\infty\right]$ by setting
\begin{displaymath}
I^+(y)=\left\{ \begin{array}{ll}
+\infty,& y\le\pi_1\\
0,& y\ge\pi_2.
\end{array}\right.
\end{displaymath}

\begin{theorem}(Theorem 3.39 \cite{Foll-Sch})
Let the random variable $X^*\in\mathcal{B}_0$ have a form $X^*=I^+(c\varphi(T))$ for such constant $c>0$ that $\E( \varphi(T) I^{+}(c\varphi(T)))=w$. If $\E u(X^*)<\infty$ then
$$\E( u(X^*))=\max\limits_{X\in\mathcal{B}_0}\E (u(X)),$$
and this maximizer is unique.
\end{theorem}
Let $u(x)=\frac{x^\gamma}{\gamma}$, $x>0$, $\gamma\in(0,1)$. Then, according to example 3.43 \cite{Foll-Sch},
$$I^+(c\varphi(T))=c^{-\frac{1}{1-\gamma}}(\varphi(T))^{-\frac{1}{1-\gamma}}.$$
If $d:=\E (\varphi(T))^{-\frac{\gamma}{1-\gamma}}<\infty$ then unique optimal profile is given by $X^*=\frac{w}{d}(\varphi(T))^{-\frac{1}{1-\gamma}}$, and the maximal value of the expected utility is equal to
$$\E( u(X^*))=\frac{1}{\gamma}w^{\gamma}d^{1-\gamma}.$$
As it was mentioned,
\begin{equation}\label{fi}
\varphi=\varphi(T)=\exp\left\{\int\limits_{0}^T\vartheta(s)dW(s)-\frac{1}{2}\int\limits_0^T \vartheta^2(s)ds,\right\}
\end{equation}
 thus $$(\varphi(T))^{-\frac{1}{1-\gamma}}=\exp\left\{-\frac{1}{1-\gamma}\int\limits_0^T \vartheta(s)dW(s)+\frac{1}{2(1-\gamma)}\int\limits_0^T \vartheta^2(s)ds\right\}.$$

Therefore, we get the following result.
\begin{theorem}
Let the process $\vartheta$ in the representation \eqref{fi} satisfy one of the conditions $(i)-(iii)$ of Lemma \ref{lem1}, and
$$\E\exp\left\{-\frac{\gamma}{1-\gamma}\int\limits_0^T\vartheta(s)dW_s+\frac{\gamma}{2(1-\gamma)}\int\limits_0^T\vartheta^2_sds\right\}<\infty.$$
Let the process $G$ satisfy the same conditions as in the Theorem \ref{TheoremIntRepresentation}. Then $X^*=\int\limits_0^T\psi(s)dG(s)$.

In the case when $u(x)=\log x$, we have $\gamma=0$ and $X^*=\frac{w}{\varphi(T)}$. Assuming that the relative entropy $H\left({\Prob}|{\Prob^*}\right)=\E(\frac{1}{\varphi(T)}\log \varphi(T))$ is a finite number, we get that
$$\E(\log X^*)=\log w + H\left({\Prob}|{\Prob^*}\right).$$
\end{theorem}
\section{Construction of the   strategy supplying the integral representation}\label{sec:6}
Consider the procedure of constructing the strategy  $\psi$, which supplies the representation \eqref{reprez}. This construction is described in the paper \cite{mish-shev}. Note, that such strategy can be non-unique. Let  $\{t_n,n\ge 1\}\in(0,1)$ be some sequence of points such that $t_n\uparrow 1$, $n\to\infty$, and $\alpha\in(1-H_2, \frac12).$   We    construct an adapted process $\psi$ such that
\begin{itemize}
\item[$(\Phi1)$] For all $n$ large enough $\int_0^{t_n}\psi_s dX_s = Z_{t_{n-1}}$.
\item[$(\Phi2)$]  $\|\psi\|_{\alpha,[t_n,1]}\to 0$, $n\to\infty$.
\end{itemize}
Since $Z_{t_n}\to Z_1$, $n\to\infty$, by continuity, these properties imply  \eqref{reprez}.
Denote for $n\ge 1$ \ $\xi_n = Z_{t_{n}}$, $\Delta_n = t_{n+1}-t_n$, $\delta_n = |\xi_{n}-\xi_{n-1}|$.

The process $\psi$ is constructed inductively on $[t_n,t_{n+1}]$. Some positive sequences $\{\sigma_n,n\ge 1\}$ and $\{\nu_n,n\ge 1\}$ such that $\sigma_n\to\infty$, $n\to\infty$ are taken.

Construction is started setting $\psi_t = 0$ for $t\in[0,t_1]$.
Further, assuming that $\psi$ is constructed on $[0,t_n)$ and denoting $V_t = \int_0^t \psi_s dX_s$, the construction is continued  depending on whether some event $A_n\in \mathcal F_{t_n}$, which will be specified later, or its complement $B_n=\Omega\setminus A_n$ holds.

Case 1:  $\omega \in A_n$. Then it was proved in \cite{mish-shev} that there exists a process $\{\phi_t,t\in [t_n,t_{n+1}]\}$ such that $\int_{t_n}^t \phi_s dX_s \to +\infty$, $t\to t_{n+1}-$.
Define $v_n = V_{t_n}-\xi_{n}$, $$\tau_n = \inf\{t\ge t_n: \int_{t_n}^t \phi_s dX_s \ge |v_n|\}$$
and set
$$
\psi_s = \phi_s  {\sign}(v_n)\mathbb{I}_{[t_n,\tau_n]}(t),\  t\in[t_n,t_{n+1}].
$$

Case 2: $\omega\in B_n$. Define $g_n(x) = \sqrt{x^2 + \nu_n^2} - \nu_n$ so that $g_n\in C^\infty(\mathbb{R})$, $|x|\ge g_n(x)\ge (|x|-\nu_{n})\vee 0$. Introduce the stopping time
$$
\tau_n = \inf\{t\ge t_n: \sigma_n g_n(X_{t}-X_{t_n})\ge \delta_n\}\wedge t_{n+1}
$$
and set
$$
\psi_s = \sigma_n g'_n(X_t-X_{t_n})  {\sign}(\xi_n-\xi_{n-1})\mathbb{I}_{[t_n,\tau_n]}(t), t\in[t_n,t_{n+1}].
$$
It is established in \cite{mish-shev} that such strategy ensures representation \eqref{reprez}.




\end{document}